\documentclass{amsart}
\usepackage{graphicx}

\newtheorem{theorem}{Theorem}[section]
\newtheorem*{theorem A}{Theorem A}
\newtheorem*{theorem B}{N\"olker's Theorem}

\theoremstyle{remark}

\theoremstyle{remark}

\theoremstyle{definition}

\newtheorem{definition}{Definition}[section]

\numberwithin{equation}{section}
\def\({\left ( }
\def\){\right )}
\def\<{\left < }
\def\>{\right >}

\begin{document}

\vspace{2cm}

\title{On the properties of new families of generalized Fibonacci numbers}

\author{Gamaliel Cerda-Morales}
\address{Centro Superior de Docencia en Ciencias B\'asicas, Escuela de Pedagog\'ia en Matem\'aticas, Universidad Austral de Chile, Puerto Montt, Chile.}
\email{gamaliel.cerda@uach.cl}




\begin{abstract}
In this paper, new families of generalized Fibonacci and Lucas numbers are introduced. In addition, we present the recurrence relations and the generating functions of the new families for $k=2$.
\end{abstract}
\maketitle

\section{Introduction}
There are a lot of integer sequences such as Fibonacci, Pell, Jacobsthal, Balancing, Lucas, Pell-Lucas, Jacobsthal-Lucas, Lucas-balancing, etc. For example, Horadam numbers are used as a generalized of Fibonacci and Lucas sequences by scientists for basic theories and their applications. For interest application of these numbers in science and nature, one can see \cite{At,Hor1,Ka-Me,Ka,Yi-Ta}. For instance, some special cases of the Horadam numbers could be derived directly using a new matrix representation, see \cite{Ce1,Ce-Mo}. 

In 1965, Horadam studied some properties of sequences of the type, $\{W_{n}\}_{n\geq0}$ or $\{W_{n}(a,b;p,q)\}$, where $a, b$ are nonnegative integers and $p, q$ are arbitrary integers, see \cite{Hor1}. Such sequences are defined by the recurrence relations of second order
\begin{equation}
W_{n}=pW_{n-1}-qW_{n-2},\ n\geq 2, 
\end{equation}
with $W_{0}=a$ and $W_{1}=b$. We are interested in the following two special cases of $\{W_{n}\}$: $\{U_{n}\}$ is defined by $U_{0}=0$, $U_{1}=1$, and $\{V_{n}\}$ is defined by $V_{0}=2$, $V_{1}=p$. It is well known that $\{U_{n}\}$ and
$\{V_{n}\}$ can be expressed in the form
\begin{equation}
U_{n}=\frac{\alpha^{n}-\beta^{n}}{\alpha-\beta},\ V_{n}=\alpha^{n}+\beta^{n},
\end{equation}
where $\alpha=\frac{p+\sqrt{\Delta}}{2}$, $\beta=\frac{p-\sqrt{\Delta}}{2}$ and the discriminant is $\Delta=p^{2}-4q$. Especially, if $p=-q=1$, $p=-2q=2$, $2p=-q=2$ and $p=6q=6$, then $\{U_{n}\}_{n\geq0}$ is the usual Fibonacci, Pell, Jacobsthal and Balancing sequence, respectively.

The aim of this work is to study some properties of two new sequences that generalize the $\{U_{n}\}_{n\geq0}$ and $\{V_{n}\}_{n\geq0}$ numbers. In this work we will follow closely the work of El-Mikkawy and Sogabe (see \cite{El-So}) where the authors give a new family that generalizes the Fibonacci numbers and establish relations with
the ordinary Fibonacci numbers.

So, in the next section we start giving the new definition of generalized Fibonacci and Lucas numbers, and we exhibit some elements of them. We also present relations of these sequences with ordinary $U_{n}$ and $V_{n}$. We deduce some properties of these new families, but using different methods. Furthermore, we study a particular case, that is two sequences of the new defined families for $k = 2$. For these sequences we present some recurrence relations and generating functions.

\section{Main Results}

We define the new generalized Fibonacci and Lucas numbers and obtain some results related to these numbers by using \cite{El-So}.
\begin{definition}\label{def:1}
Let $n$ and $k\neq 0$ be natural numbers. There exist unique numbers $m$ and $r$ such that $n=mk+r$ ($0\leq r<k$). Using these parameters, we define new generalized Fibonacci numbers $U_{n}^{(k)}$ and generalized Lucas numbers $V_{n}^{(k)}$ by 
\begin{equation}\label{eq1}
U_{n}^{(k)}=\frac{1}{(\alpha -\beta )^{k}}(\alpha ^{m+1}-\beta^{m+1})^{r}(\alpha ^{m}-\beta ^{m})^{k-r}  
\end{equation}
and
\begin{equation}\label{eq2}
V_{n}^{(k)}=(\alpha ^{m+1}+\beta ^{m+1})^{r}(\alpha ^{m}+\beta ^{m})^{k-r}
\end{equation}
where $\alpha =\frac{p+\sqrt{\bigtriangleup }}{2}$ and $\beta =\frac{p-\sqrt{\bigtriangleup }}{2}$, respectively.
\end{definition}

The first few numbers of the new generalized Fibonacci and generalized Lucas
numbers as follows
\begin{equation}
\{U_{n}^{(1)}\}_{n=0}^{3}=\{0,1,p,p^{2}-q\},\text{ }\{V_{n}^{(1)}\}_{n=0}^{3}=\{2,p,p^{2}-2q,p^{3}-3pq\},
\end{equation}
\begin{equation}
\{U_{n}^{(2)}\}_{n=0}^{3}=\{0,0,1,p\},\text{ }\{V_{n}^{(2)}\}_{n=0}^{3}=\{4,2p,p^{2},p^{3}-2pq\}, 
\end{equation}
\begin{equation}
\{U_{n}^{(3)}\}_{n=0}^{3}=\{0,0,0,1\},\text{ }\{V_{n}^{(3)}\}_{n=0}^{3}=\{8,4p,2p^{2},8\}
\end{equation}
for $k=1,2,3$.

Given $n$ a nonnegative integer and $k$ a natural number, then $$U_{mk+r}^{(k)}=U_{m}^{k-r}U_{m+1}^{r}\ \textrm{and}\ V_{mk+r}^{(k)}=V_{m}^{k-r}V_{m+1}^{r},$$ where $m$ and $r$ are nonnegative integers such that $n= mk + r$ ($0\leq r < k$). Note that $U_{n}^{(1)}$ is the generalized Fibonacci numbers $U_{n}$ and $V_{n}^{(1)}$ is the generalized Lucas numbers $V_{n}$.

In the following, we give some properties of the new generalized Fibonacci and generalized Lucas numbers.
\begin{theorem}
Let $k,m\in \{1,2,3,...\}$ be fixed numbers. The new generalized Fibonacci numbers and generalized Lucas numbers satisfy

\begin{enumerate}
\item[(i)] $\sum_{t=0}^{k-1}\binom{k-1}{t}(-p)^{-t}U_{mk+t}^{(k)}=\left( \frac{q}{p}\right) ^{k-1}U_{m}U_{(m-1)(k-1)}^{(k-1)},$
\item[(ii)] $\sum_{t=0}^{k-1}\binom{k-1}{t}(\frac{-p}{q})^{t}U_{mk+t}^{(k)}=(-q)^{1-k}U_{m}U_{(m+2)(k-1)}^{(k-1)}$ and
\item[(iii)] $\sum_{t=0}^{k-1}p^{-t}U_{mk+t}^{(k)}=\left( -\frac{p}{q}\right) p^{-k}\left( \frac{U_{m}}{U_{m-1}}\right) \left(
U_{(m+1)k}^{(k)}-p^{k}U_{mk}^{(k)}\right)$.
\end{enumerate}

\begin{proof}
Note of the definition 1, $U_{n}^{(k)}=U_{m}^{k-r}U_{m+1}^{r},$ where $n=mk+r$ and $0\leq r<k$. Then, (i). 
\begin{align*}
\sum_{t=0}^{k-1}\binom{k-1}{t}(-p)^{-t}U_{mk+t}^{(k)}&=(-p)^{1-k}U_{m}\sum_{t=0}^{k-1}\binom{k-1}{t}(-p)^{k-1-t}U_{m}^{k-1-t}U_{m+1}^{t} \\
&=(-p)^{1-k}U_{m}\sum_{t=0}^{k-1}\binom{k-1}{t}(-pU_{m})^{k-1-t}U_{m+1}^{t}.
\end{align*}
Using the binomial theorem, $\sum_{t=0}^{k-1}\binom{k-1}{t}(-pU_{m})^{k-1-t}U_{m+1}^{t}=(-pU_{m}+U_{m+1})^{k-1}$. Then, 
\begin{align*}
\sum_{t=0}^{k-1}\binom{k-1}{t}(-p)^{-t}U_{mk+t}^{(k)}&=(-p)^{1-k}U_{m}(-pU_{m}+U_{m+1})^{k-1} \\
&=(-p)^{1-k}U_{m}(-qU_{m-1})^{k-1} \\
&=\left( \frac{q}{p}\right) ^{k-1}U_{m}U_{(m-1)(k-1)}^{(k-1)},
\end{align*}
is obtained and this completes the proof.

(ii). 
\begin{align*}
\sum_{t=0}^{k-1}\binom{k-1}{t}(\frac{-p}{q})^{t}U_{mk+t}^{(k)}&=U_{m}\sum_{t=0}^{k-1}\binom{k-1}{t}(-q)^{-t}p^{t}U_{m}^{k-1-t}U_{m+1}^{t} \\
&=(-q)^{1-k}U_{m}\sum_{t=0}^{k-1}\binom{k-1}{t}(-qU_{m})^{k-1-t}(pU_{m+1})^{t}.
\end{align*}
Using the binomial theorem, $\sum_{t=0}^{k-1}\binom{k-1}{t}(-qU_{m})^{k-1-t}(pU_{m+1})^{t}=(pU_{m+1}-qU_{m})^{k-1}.$ Then, 
\begin{align*}
\sum_{t=0}^{k-1}\binom{k-1}{t}(\frac{-p}{q})^{t}U_{mk+t}^{(k)}&=(-q)^{1-k}U_{m}(pU_{m+1}-qU_{m})^{k-1} \\
&=(-q)^{1-k}U_{m}U_{m+2}^{k-1} \\
&=(-q)^{1-k}U_{m}U_{(m+2)(k-1)}^{(k-1)},
\end{align*}
is obtained.

(iii). 
$
\sum_{t=0}^{k-1}p^{-t}U_{mk+t}^{(k)}=\sum_{t=0}^{k-1}p^{-t}U_{m}^{k-t}U_{m+1}^{t}=U_{m}^{k}\sum_{t=0}^{k-1}\left( \frac{U_{m+1}}{pU_{m}}\right) ^{t}$,
but $$\sum_{t=0}^{k-1}\left( \frac{U_{m+1}}{pU_{m}}\right) ^{t}=\frac{\left( \frac{U_{m+1}}{pU_{m}}\right) ^{k}-1}{\frac{U_{m+1}}{pU_{m}}-1}=(pU_{m})^{1-k}\left( \frac{U_{m+1}^{k}-pU_{m}}{U_{m+1}-pU_{m}}\right).$$
Then,
\begin{align*}
\sum_{t=0}^{k-1}p^{-t}U_{mk+t}^{(k)}&=U_{m}^{k}(pU_{m})^{1-k}\left( \frac{U_{m+1}^{k}-p^{k}U_{m}^{k}}{U_{m+1}-pU_{m}}\right)  \\
&=p^{1-k}U_{m}\left( \frac{U_{m+1}^{k}-p^{k}U_{m}^{k}}{-qU_{m-1}}\right)  \\
&=\left( -\frac{p}{q}\right) p^{-k}\left( \frac{U_{m}}{U_{m-1}}\right)\left( U_{(m+1)k}^{(k)}-p^{k}U_{mk}^{(k)}\right).
\end{align*}
\end{proof}
\end{theorem}

\section{A particular case $U_{n}^{(2)}$ and $V_{n}^{(2)}$}

In \cite{Ce1}, the author use a matrix method for generating the generalized Fibonacci numbers by defining the $W(p,q)$-matrix and proved that
\begin{equation}
W(p,q)^{n}=\left[ 
\begin{array}{cc}
p & -q \\ 
1 & 0
\end{array}
\right]^{n}=\left[ 
\begin{array}{cc}
U_{n+1} & -qU_{n} \\ 
U_{n} & -qU_{n-1}
\end{array}
\right],
\end{equation}
for any natural number $n$. 

Thus, for any $n,l\geq 0$ and $n+l\geq 2$, we have
$$
\left[ 
\begin{array}{cc}
U_{n+l} & -qU_{n+l-1} \\ 
U_{n+l-1} & -qU_{n+l-2}
\end{array}
\right]=W(p,q)^{n+l-2}\left[ 
\begin{array}{cc}
p & -q \\ 
1 & 0
\end{array}
\right].
$$
If we compute the determinant of both sides of the previous equality we obtain
\begin{equation}
-qU_{n+l}U_{n+l-2}+qU_{n+l-1}^{2}=q\det(W(p,q)^{n+l-2}).
\end{equation}
which is equivalent to $$U_{n+l-1}^{2}-U_{n+l}U_{n+l-2}=q^{n+l-2}.$$

Since, by Definition \ref{def:1} (where $m =n+l-1$, $k = 2$ and $r = 0$) $$U_{2(n+l-1)}^{(2)}=U_{n+l-1}^{2}.$$
We have proved the following result:
\begin{theorem}\label{th:3}
We have the following relation for the new family of generalized Fibonacci numbers $\{U_{n}^{(2)}\}$ 
\begin{equation}
U_{2(n+l-1)}^{(2)}=q^{n+l-2}+U_{n+l}U_{n+l-2},
\end{equation}
where $n+l\geq 2$.
\end{theorem}

We study the particular case of the sequences $\{U_{n}^{(2)}\}_{n\geq0}$ and $\{V_{n}^{(2)}\}_{n\geq0}$ defined by \ref{eq1} and \ref{eq2}, respectively. First, we present a recurrence relation for these sequences.

\begin{theorem}\label{th:4}
The sequences $\{U_{n}^{(2)}\}$ and $\{V_{n}^{(2)}\}$ satisfy, respectively, the following recurrence relations: 
\begin{equation}
U_{n}^{(2)}=pU_{n-1}^{(2)}-pqU_{n-3}^{(2)}+q^{2}U_{n-4}^{(2)},\ n\geq4,
\end{equation}
and
\begin{equation}
V_{n}^{(2)}=pV_{n-1}^{(2)}-pqV_{n-3}^{(2)}+q^{2}V_{n-4}^{(2)},\ n\geq4,
\end{equation}
\end{theorem}
\begin{proof}
There are two cases of subscript $n$. Here we use two relations for the proof $U_{2m}^{(2)}=U_{m}^{2}$ and $U_{2m+1}^{(2)}=U_{m}U_{m+1}$. Then, by using these relations we have the following 
\begin{align*}
U_{2m}^{(2)}&=U_{m}^{2} \\
&=U_{m}\left( pU_{m-1}-qU_{m-2}\right)  \\
&=pU_{m}U_{m-1}-q\left( pU_{m-1}-qU_{m-2}\right) U_{m-2} \\
&=pU_{2m-1}^{(2)}-pqU_{2m-3}^{(2)}+q^{2}U_{2m-4}^{(2)}.
\end{align*}

If we take $n=2m$ we have the following equation $U_{n}^{(2)}=pU_{n-1}^{(2)}-pqU_{n-3}^{(2)}+q^{2}U_{n-4}^{(2)}$. If $n$ is odd integer, then we have
\begin{align*}
U_{2m+1}^{(2)} &=U_{m}U_{m+1} \\
&=U_{m}\left( pU_{m}-qU_{m-1}\right)  \\
&=pU_{m}U_{m}-q\left( pU_{m-1}-qU_{m-2}\right) U_{m-1} \\
&=pU_{2m}^{(2)}-pqU_{2m-2}^{(2)}+q^{2}U_{2m-3}^{(2)}.
\end{align*}
In a similar way we can prove the result for $\{V_{n}^{(2)}\}_{n\geq0}$. This theorem is completed.
\end{proof}

We also note that if we consider separately the even and the odd terms of the above defined sequences we can obtain shorter recurrence relations. In fact, for $n = 2m$, for any natural number $m$, by Theorem \ref{th:3} (with $n = m$ and $l=1$) we have
\begin{align*}
U_{2m}^{(2)} &=U_{m+1}U_{m-1}+q^{m-1} \\
&=U_{m-1}\left( pU_{m}-qU_{m-1}\right) +q^{m-1} \\
&=pU_{m-1}U_{m}-qU_{m-1}^{2}+q^{m-1} \\
&=pU_{2m-1}^{(2)}-qU_{2m-2}^{(2)}+q^{m-1}.
\end{align*}

Similarly, for the odd case $n=2m+1$, we have
\begin{align*}
U_{2m+1}^{(2)} &=U_{m}U_{m+1} \\
&=pU_{m}U_{m}-qU_{m}U_{m-1} \\
&=pU_{2m}^{(2)}-qU_{2m-1}^{(2)}.
\end{align*}
Therefore we can conclude the following. A shorter recurrence relation for the sequence $\{U_{n}^{(2)}\}$ is given by
$$\left\{
\begin{array}{c l}
 U_{2m}^{(2)}=pU_{2m-1}^{(2)}-qU_{2m-2}^{(2)}+q^{m-1}\\
 U_{2m+1}^{(2)}=pU_{2m}^{(2)}-qU_{2m-1}^{(2)}
\end{array}
\right .
$$
for the even and the odd terms.

Next we find generating functions for these sequences. 
\begin{theorem}
The generating function $g^{(2)}(x)$ for $\{U_{n}^{(2)}\}$ is given by 
\begin{equation}
g^{(2)}(x)=\frac{x^{2}}{1-px+pqx^{3}-q^{2}x^{4}}.
\end{equation}
\end{theorem}

\begin{proof}
To the sum of this power series, $g^{(2)}(x)=\sum_{n=0}^{\infty }U_{n}^{(2)}x^{n}$, we call generating function of the new generalized Fibonacci sequence of numbers $\{U_{n}\}_{n\geq0}$. 
Then,
\begin{align*}
(1-px+&pqx^{3}-q^{2}x^{4})g^{(2)}(x)\\
&=(U_{0}^{(2)}+U_{1}^{(2)}x+U_{2}^{(2)}x^{2}+U_{3}^{(2)}x^{3}+\cdots)\\
&\ \ -(pU_{0}^{(2)}x+pU_{1}^{(2)}x^{2}+pU_{2}^{(2)}x^{3}+pU_{3}^{(2)}x^{4}+\cdots)\\
&\ \ +(pqU_{0}^{(2)}x^{3}+pqU_{1}^{(2)}x^{4}+pqU_{2}^{(2)}x^{5}+pqU_{3}^{(2)}x^{6}+\cdots)\\
&\ \ -(q^{2}U_{0}^{(2)}x^{4}+q^{2}U_{1}^{(2)}x^{5}+q^{2}U_{2}^{(2)}x^{6}+q^{2}U_{3}^{(2)}x^{7}+\cdots)\\
&=U_{0}^{(2)}+(U_{1}^{(2)}-pU_{0}^{(2)})x+(U_{2}^{(2)}-pU_{1}^{(2)})x^{2}+(U_{3}^{(2)}-pU_{2}^{(2)}+pqU_{0}^{(2)})x^{3}.
\end{align*}

Hence, taking into account the initial conditions of the sequence $\{U_{n}^{(2)}\}_{n\geq0}$ and by Theorem \ref{th:4}, this is equivalent to
$$(1-px+pqx^{3}-q^{2}x^{4})g^{(2)}(x)=x^{2}$$
and therefore
$$g^{(2)}(x)=\frac{x^{2}}{1-px+pqx^{3}-q^{2}x^{4}},$$
which completes the proof.
\end{proof}
\begin{theorem}
The generating function $h^{(2)}(x)$ for $\{V_{n}^{(2)}\}$ is given by 
\begin{equation}
h^{(2)}(x)=\frac{4-2px-p^{2}x^{2}+2pqx^{3}}{1-px+pqx^{3}-q^{2}x^{4}}.
\end{equation}
\end{theorem}
\begin{proof}
The proof can be obtained in a similar way to the proof of above theorem. 
\end{proof}

\end{document}